\documentclass[10pt]{amsart}
\usepackage[utf8]{inputenc}

\usepackage{amssymb,amsthm,amsmath, mathtools}
\usepackage{mathrsfs,bbm}
\usepackage{enumerate}
\usepackage{graphicx,xcolor}
\usepackage{setspace}
\usepackage[hidelinks]{hyperref}
\usepackage{cancel}

\newcommand{\dd}{\mathrm{d}}
\newcommand{\E}{\mathbb{E}}
\newcommand{\pp}{\mathbb{P}}
\newcommand{\1}{\textbf{1}}
\newcommand{\R}{\mathbb{R}}

\newcommand{\p}[1]{\mathbb{P}\left( #1 \right)}
\newcommand{\scal}[2]{\left\langle #1, #2 \right\rangle}
\newcommand{\red}{}

\DeclareMathOperator{\Var}{Var}

\DeclareMathOperator{\vol}{vol}

\usepackage{xpatch}
\makeatletter
\def\thm@space@setup{%
  \thm@preskip=12pt plus 0pt minus 0pt
  \thm@postskip=0pt plus 0pt minus 0pt
}
\xpatchcmd{\proof}{6\p@\@plus6\p@\relax}{\z@skip}{}{}
\makeatother

\newtheorem{theorem}{Theorem}
\newtheorem{lemma}[theorem]{Lemma}
\newtheorem{corollary}[theorem]{Corollary}
\newtheorem{proposition}[theorem]{Proposition}

\theoremstyle{remark}
\newtheorem{remark}[theorem]{Remark}

\theoremstyle{definition}

\title{A R\'enyi entropy interpretation of anti-concentration and noncentral sections of convex bodies}

\author{James Melbourne}

\address{(JM) Department of Probability and statistics, Centro de Investigacion en matem\'aticas (CIMAT), Mexico.}

\author{Tomasz Tkocz}

\address{(TT) Carnegie Mellon University; Pittsburgh, PA 15213, USA.}

\author{Katarzyna Wyczesany}

\address{(KW) {\red School of Mathematics, University of Leeds, Leeds LS2 9JT, UK}}

\email{ttkocz@math.cmu.edu}

\thanks{TT's research supported in part by NSF grant DMS-2246484.}

\begin{document}

\begin{abstract} 
We extend Bobkov and Chistyakov's  (2015) upper bounds on concentration functions of sums of independent random variables to a multivariate entropic setting. The approach is based on pointwise estimates on densities of sums of independent random vectors uniform on centred Euclidean balls. In this vein, we also obtain sharp bounds on volumes of noncentral sections of isotropic convex bodies.
\end{abstract}

\maketitle

\bigskip

\begin{footnotesize}

\end{footnotesize}

\bigskip

\section{Introduction}

Anti-concentration is a phenomenon which asserts that random variables have a ``small" probability of falling within a certain range, or in other words, it quantifies the scatter of the values of the random variable. In particular, one is interested in the rate of increase of anti-concentration of a sum of independent random variables, which we further address in this note in an entropic multivariate setting.  

More precisely, for a random vector $X$ taking values in $\R^d$, we define its \emph{concentration function} $Q_X \colon [ 0,\, +\infty) \to [0,1]$ as
\[
Q_X(\lambda) = \sup_{x \in \R^d} \p{|X-x| \leq \lambda}, \qquad \lambda \geq 0,
\]
where $|\cdot|$ is the standard Euclidean norm on $\R^d$. The anti-concentration phenomenon has been quantified in a number of classical results, and can be traced back to works of Doeblin, L\'evy, Kolmogorov, \cite{Do, Kol, Le}. Rogozin's inequality from \cite{Rog} strengthened all those and it states that there is a universal positive constant $C$ such that for independent random variables $X_1, \dots, X_n$, their sum $S = X_1 + \dots + X_n$ and positive parameters $\lambda_1, \dots, \lambda_n$, we have
\[
Q_S(\lambda) \leq C{\red \lambda}\left(\sum_{j=1}^n \lambda_j^2\Big(1-Q_{X_j}(\lambda_j)\Big)\right)^{-1/2}, \qquad \lambda \geq \max_{j\leq n} \lambda_j.
\]
Esseen in \cite{Ess} offered an analytic approach based on characteristic functions. This led to further improvements, by Kesten in \cite{Kes1, Kes2}, as well as Postnikova and Yudin in \cite{PY}, culminating in a bound improving upon all previous ones, established by Miroshnikov and Rogozin in \cite{MRog}, which gives
\[
Q_S(\lambda) \leq C{\red \lambda}\left(\sum_{j=1}^n \lambda_j^2D_{X_j}(\tfrac12\lambda_j)Q_{X_j}(\lambda_j)^{-2}\right)^{-1/2}, \qquad \lambda \geq \frac12\max_{j\leq n} \lambda_j,
\]
where $D_X(\lambda) = \lambda^{-2}\, \E[\min\{|X|, \lambda\}^2]$. Note that $D_X \leq 1$. Recently, Bobkov and Chistyakov in \cite{BCh} have further strengthened this inequality by removing the factors $D_{X_j}$ at the expense of shrinking the domain $\lambda \gtrsim \max \lambda_j$ to $\lambda \gtrsim \left(\sum \lambda_j^2\right)^{1/2}$, which is necessary for such a modified inequality to hold (see their remark before Theorem 1.2 in \cite{BCh}). Namely, they obtain the inequality
\begin{equation}\label{eq:BCh}
Q_S(\lambda) \leq C{\red \lambda}\left(\sum_{j=1}^n \lambda_j^2Q_{X_j}(\lambda_j)^{-2}\right)^{-1/2}, \qquad \lambda \geq \left(\textstyle\sum_{j=1}^n \lambda_j^2\right)^{1/2},
\end{equation}
with a universal positive constant $C$.
They were motivated by two-sided bounds on the concentration function of sums of log-concave random variables. Crucially for their approach, they have obtained  a uniform bound on the density of the sum of independent uniform random variables. {\red This bound can be naturally restated in geometric terms as the statement that the volume of \emph{every} section of the cube $[-1,1]^n$ by a hyperplane at distance at most $1$ from the origin is large: at least a universal fraction of the volume of the cube.}

The aim of this note is to extend these results to higher dimensions, as well as provide a new extension of those to R\'enyi entropies, which continues the recent body of works devoted to developing subadditivity properties for sums of independent random variables in various settings, see for instance \cite{BCh-Renyi, BM, BMM, MMR}. Our approach has incidentally led us to a curious sharp lower bound on noncetral sections of isotropic convex bodies, which may be of independent interest.

\subsection{Noncentral sections}

Our first main result is the following uniform bound.

\begin{theorem}\label{thm:low-bd}
Let $d \geq 1$. Let $U_1, U_2, \dots$ be i.i.d. random vectors uniform on the unit Euclidean ball $B_2^d$ in $\R^d$. There is a positive constant $c_d$ depending only on $d$ such that for every $n \geq 1$ and real numbers $a_1, \dots, a_n$ with $\sum_{j=1}^n a_j^2 = 1$, we have
\[
\inf_{x \in B_2^d} p(x) \geq c_d,
\]
where $p$ is the density of the random vector $\,\sum_{j=1}^n a_jU_j$.
\end{theorem}

In the $1$-dimensional case $d=1$, this was discovered by Bobkov and Chistyakov in \cite{BCh} (Proposition 3.2), as alluded to earlier (they obtained $c_1 = 0.00095..$). Motivated by applications to noncentral sections of the cube and polydisc, K\"onig and Rudelson in \cite{KR} studied the cases $d=1$ and $d=2$ and obtained that one can take $c_1 = \frac{1}{34} = 0.029..$ and $c_2 =  \frac{1}{27\pi} = 0.011..$. As we shall present, without too much additional work, their probabilistic approach essentially yields the claimed result for arbitrary $d$ with
\begin{equation}\label{eq:const-c_d}
c_d = \frac{1}{100\cdot2^d\omega_d},
\end{equation}
where as usual $\omega_d$ stands for the volume of the unit ball in $\R^d$,
\[
\omega_d = \vol_d(B_2^d) = \frac{\pi^{d/2}}{\Gamma(\tfrac{d}{2}+1)}.
\]

Pursuing a more geometric direction, we extend the Bobkov-Chistyakov result to a sharp bound for all even log-concave densities. 
Recall that a random vector $X$ in $\R^d$ with density $f$ is called log-concave when $f = e^{-\phi}$ for a convex function $\phi\colon \R^d \to {\red (-\infty,+\infty]}$ (for background, see for instance \cite{AGM}). 

\begin{theorem}\label{thm:low-bd-at-sqrt3}
Let $f\colon \R \to [0,+\infty)$ be an even log-concave probability density. Let 
\[
\sigma = \sqrt{\int_{\R} x^2f(x) \dd x}
\]
be its variance. Then
\begin{equation}\label{eq:f-at-sigma}
\sigma f(\sigma \sqrt{3}) \geq \frac{1}{\sqrt{2}}e^{-\sqrt{6}} = 0.061...
\end{equation}
(The equality is attained for the symmetric exponential density.)
\end{theorem}

The example of the symmetric uniform distribution shows that in the parameter $\sigma\sqrt{3}$, constant $\sqrt{3}$ cannot be replaced with any larger number for such a lower bound to continue to hold (uniformly over all even log-concave densities). The parameter $\sigma\sqrt{3}$ can be loosely thought of as the \emph{effective support} of $f$, as  stems from the following basic lemma (see, e.g. Theorem 5 in \cite{MNT} for a generalisation to R\'enyi entropies). {\red Note that every even log-concave probability density on $\R$ is, in particular, nonincreasing on $[0,+\infty)$.}

\begin{lemma}\label{lm:eff-supp}
Let $f\colon \R \to [0,+\infty)$ be an even probability density of variance $1$. {\red Suppose that $f$ is nonincreasing on $[0,+\infty)$.} Then the support of $f$, that is the set $\,\mathrm{supp}(f) = \overline{\{x \in \R, \ f(x) > 0\}}$ contains the interval $[-\sqrt{3},\sqrt{3}]$.
\end{lemma}
\begin{proof}
Suppose that  $\mathrm{supp}(f) = [-a,a]$ with $a < \sqrt{3}$. Let $g(x) = \frac{1}{2\sqrt{3}}\1_{[-\sqrt3,\sqrt3]}(x)$ be the uniform density on $[-\sqrt3,\sqrt3]$ with variance $1$. By the monotonicity of $f$, $f$ intersects $g$ on $[0,+\infty)$ at a point $c \in [0,a]$, and $f-g \geq 0$ on $[0,c]$, $f-g \leq 0$ on {\red $[c,+\infty)$}. We get
\begin{align*}
0 = \int_0^\infty x^2(f(x)-g(x)) \dd x  \ &{\red <} \ c^2\int_0^c (f(x) - g(x)) \dd x + c^2{\red \int_c^\infty} (f(x) - g(x)) \dd x  \\
&= c^2\int_0^\infty (f-g) = 0,
\end{align*}
a contradiction {\red (the inequality is strict because in the second integral we have $f-g = -g = -\frac{1}{2\sqrt3}$ on $(a, \sqrt3)$).}
\end{proof}

Theorem \ref{thm:low-bd-at-sqrt3} readily yields a sharp lower bound for the volume of noncentral sections of isotropic symmetric convex bodies (on their \emph{effective support}). For a recent survey on this topic, see \cite{NT}. Recall that a convex body $K$ in $\R^d$ is called (centrally) \emph{symmetric} if $K = -K$ and in that special case it is called \emph{isotropic} if it has volume~$1$ and covariance matrix proportional to the identity matrix, \[
\left[\int_K x_ix_j \dd x\right]_{i,j \leq d} = L_K^2I_{d \times d},
\]
and the proportionality constant $L_K > 0$ is called the \emph{isotropic constant} of $K$.

\begin{corollary}\label{cor:noncent}
Let $K$ be a symmetric isotropic convex body in $\R^d$ with isotropic constant $L_K$. For every hyperplane $H$ in $\R^d$ with distance at most $L_K\sqrt{3}$ to the origin, we have
\begin{equation}\label{eq:noncent}
\vol_{d-1}(K \cap H) \geq \frac{1}{L_K}\frac{1}{\sqrt{2}}e^{-\sqrt{6}}.
\end{equation}
\end{corollary}

\begin{remark}\label{rem:noncent-sharp}
This bound is sharp, in that for every $\epsilon > 0$, there is $d_\epsilon$ and a symmetric isotropic convex body $K$ in $\R^{d_\epsilon}$ which admits a hyperplane $H$ at distance $L_{K}\sqrt{3}$ to the origin for which $\vol_{d-1}(K \cap H) < \frac{1}{L_K}\left(\frac{1}{\sqrt{2}}e^{-\sqrt{6}}+\epsilon\right).$
\end{remark}

{\red
\begin{remark}
In light of the very recent breakthrough of Klartag and Lehec from \cite{KlaLeh}, for all bodies $K$ in all dimensions, $L_K \leq C$ for a universal constant $C$. Moreover, it is well known that $L_K \geq \frac{1}{12}$ (see Section \ref{sec:reversals} for further details). As a result, every section of every isotropic convex body $K$ by a hyperplane at distance at most $\frac{\sqrt3}{12}$ (from the origin)  has volume bounded away from $0$ by a universal positive constant.
\end{remark}

\begin{remark}
When specialised to the unit volume cube $K = [-\frac12, \frac12]^d$ for which $L_K = \frac{1}{\sqrt{12}}$, we obtain that every section of $K$ by a hyperplane at distance at most $\frac12$ has volume at least $\sqrt{6}e^{-\sqrt6} = 0.21..$, i.e. we can take $c_1 = \sqrt{6}e^{-\sqrt6}$ in Theorem~\ref{thm:low-bd}, improving on the value of the numerical constant from \cite{KR}.
\end{remark}
}

\subsection{Subadditivity of R\'enyi entropy}

To state our second main result and elucidate on the connection between R\'enyi entropies and concentration function, we begin with recalling the necessary definitions. 

For a random vector $X$ in $\R^d$ with density $f$ on $\R^d$, and $p \in [0,+\infty]$, we define the $p$\emph{-R\'enyi entropy} of $X$ as 
    \[
    h_p(X) = \frac{1}{1-p}\log\left(\int_{\R^d} f(x)^p \dd x\right),
    \]
with the cases $p \in \{0,1,\infty\}$ \,treated by limiting expressions: $h_0(X) = \log \vol_d(\text{supp}(f))$, $h_1(X) = - \int_{\R^d}  f \log f$, and $h_\infty(X) \coloneqq - \log \|f\|_\infty$, provided the relevant integrals exist (in the Lebesgue sense).  We define the \emph{R\'enyi entropy power} to be
    \[
        N_p(X) = e^{2 h_p(X)/d}.
    \]
Finally,  the \emph{maximum functional} $M$ for $X$ is defined by
    \[
        M(X) = \|f\|_\infty
    \]
and we have
\[
N_\infty(X) = M(X)^{-2/d}.
\]
We observe that if $U$ is uniform on the unit ball $B_2^d$ and independent of $X$, then the concetration function of $X$ is up to scaling factors the maximum functional of the smoothed variable $X+\lambda U$, that is, plainly
\begin{equation}\label{eq:Q-M}
\begin{split}
Q_X(\lambda) = \sup_{x \in \R^d} \p{|X-x|\leq \lambda} &= \sup_{x \in \R^d} \int_{\R^d} \1_{\{|y-x|\leq \lambda\}} f(y) \dd y \\
&= \lambda^d\omega_d \,M(X+ \lambda U)
\end{split}
\end{equation}
and, consequently,
\begin{equation}\label{eq:Q-N}
N_\infty(X+\lambda U) = \omega_d^{2/d}\lambda^2 \,Q_X(\lambda)^{-2/d}.
\end{equation}
This relationship allows to rewrite the anti-concentration bound \eqref{eq:BCh} in terms of the maximum functional, or $\infty$-R\'enyi entropy power, of smoothed densities. It turns out that thanks to Theorem \ref{thm:low-bd}, the same continues to hold for $p$-R\'enyi entropies of random vectors in $\R^d$.

\begin{theorem} \label{thm:Renyi-anti-conc}
Let $p > 1$. For all independent random vectors $X_1, \dots, X_n$ in $\R^d$, their sum $S = \sum_{j=1}^n X_j$ and positive parameters $\lambda_1, \dots, \lambda_n$ with $\sum_{j=1}^n\lambda_j^2= 1$, we have
    \begin{equation}
        N_p\left(S +U_0\right) \geq \frac{1}{C_{p,d}}\sum_{j=1}^n N_p(X_j + \lambda_j U_j),
    \end{equation}
where $U_0, U_1, \dots, U_n$ are independent random vectors uniform on the unit ball in $\R^d$, also independent of the $X_j$'s. One can take $C_{p,d} = e\cdot 2^{\frac{2p}{p-1}\frac{d+7}{d}}$.
\end{theorem}


As a corollary, we get an extension of \eqref{eq:BCh} to multivariate random variables.

\begin{corollary}\label{cor:anti-conc}
For all independent random vectors $X_1, \dots, X_n$ in $\R^d$, their sum $S = X_1 + \dots + X_n$, positive parameters $\lambda_1, \dots, \lambda_n$ and $\lambda \geq \left(\sum_{j=1}^n\lambda_j^2\right)^{1/2}$, we have
    \[
       Q_X(\lambda) \leq {\red \Big(2\lambda+ (\textstyle\sum \lambda_j^2)^{1/2}\Big)^d \,e^{d/2}2^{d+7}}\left(  \sum_{j=1}^n \lambda_j^2 \ Q_{X_j}(\lambda_j)^{-2/d} \right)^{-d/2}.
    \]
\end{corollary}

The next three sections contain the proofs of our main results Theorem \ref{thm:low-bd}, \ref{thm:low-bd-at-sqrt3} and \ref{thm:Renyi-anti-conc}. The last section is devoted to remarks on reverse bounds in the log-concave setting.

{\red 
\subsection*{Acknowledgements.} 
We are immensely indebted to an anonymous referee for their very careful reading of the manuscript and numerous invaluable suggestions, in particular for Remark \ref{rem:caveat}.}

\section{Sums of uniforms: Proof of Theorem \ref{thm:low-bd}}

Throughout this section we fix $d \geq 1$, let $U_1, U_2, \dots$ be  i.i.d. random vectors uniform on the Euclidean unit ball $B_2^d$ in $\R^d$ and let $\xi_1, \xi_2, \dots$ be i.i.d. random vectors uniform on the Euclidean unit sphere $S^{d+1}$ in $\R^{d+2}$. We also fix $n \geq 1$ and real numbers $a_1, \dots, a_n$ with $\sum_{j=1}^n a_j^2 = 1$. Theorem~\ref{thm:low-bd} holds trivially for $n=1$. Thus we shall assume in all the statements of this section that $n \geq 2$ with all the $a_j$ nonzero.

\subsection{A probabilistic formula}

One of the key ingredients is the following probabilistic formula for the density $p$ of $\sum_{j=1}^n a_jU_j$, established in \cite{KR} via a delicate Fourier analytic argument when $d=1, 2$ (Proposition 3.2 in \cite{KR}). We extend it to all dimensions and give an elementary direct and short proof.

\begin{lemma}\label{lm:density}
For every $x \in \R^d$, we have
\[
p(x) = \frac{1}{\omega_d} \,\E\Big[\Big|\sum_{j=1}^n a_j\xi_j\Big|^{-d}\1_{\left\{\left|\sum_{j=1}^n a_j\xi_j\right| > |x|\right\}}\Big].
\]
\end{lemma}

The crux is an intimate connection between the uniform measure on the sphere and its projection to a codimension $2$ subspace which turns out to be uniform on the ball. This is folklore, and when specialised to two dimensional spheres amounts to the Archimedes' Hat-Box theorem.  We refer to Corollary 4 in \cite{BGMN} for a generalisation to $\ell_p$ balls.

\begin{lemma}\label{lm:Archimedes}
Let $d \geq 1$ and let $X = (X_1, \dots, X_d, X_{d+1}, X_{d+2})$ be a random vector uniform on the unit Euclidean sphere $S^{d+1}$ in $\R^{d+2}$. The random vector $\tilde X = (X_1, \dots, X_d)$  in  $\R^d$ is uniform on the unit Euclidean ball $B_2^d$.
\end{lemma}
\begin{proof}
Let $P\colon S^{d+1} \to B_2^d$ be the projection map $P(x_1, \dots, x_d,x_{d+1},x_{d+2}) = (x_1, \dots, x_d)$. The preimage of a point $x \in B_2^d$ with $|x| = r$ under $P$ is a circle $x_{d+1}^2+x_{d+2}^2 = 1-r^2$ of radius $\sqrt{1-r^2}$. Using cylindrical coordinates $(r,x_{d+1},x_{d+2})$, the preimage on $S^{d+1}$ of an infinitesimal volume element $dr$ under $P$ then has $(d+1)$-volume on $S^{d+1}$ equal to 
\[
2\pi \sqrt{1-r^2}\sqrt{(\dd (\sqrt{1-r^2}))^2+(\dd r)^2} = 2\pi \dd r,
\]
which is uniform on $B_2^d$ (i.e. does not depend on $r$). 
\end{proof}

\begin{proof}[Proof of Lemma \ref{lm:density}]
Let $X = \sum_{j=1}^n a_j\xi_j$, $Y = \sum_{j=1}^n a_jU_j$ and let $P\colon S^{d+1} \to B_2^d$ be the projection map $P(x_1, \dots, x_d,x_{d+1},x_{d+2}) = (x_1, \dots, x_d)$. Note that
$
P(X) = \sum_{j=1}^n a_jP(\xi_j),
$ and by Lemma \ref{lm:Archimedes}, each $P(\xi_j)$ has the same distribution as $U_j$. Therefore, $Y$ has the same distribution as $P(X)$. For a Borel set $A$ in $\R^d$ we thus have,
\[
\p{Y \in A} = \p{P(X) \in A} = \p{X \in A\times \R^2}.
\]
Since $X$ is rotationally invariant, we can write $X = |X|\theta$, where $\theta$ is a random vector uniform on $S^{d+1}$, independent of $X$. Using this independence, we condition on the values of $X$ and continue the calculation as follows
\[
\p{X \in A\times \R^2} = \E_X\pp_\theta\left(\theta \in \frac{1}{|X|}(A\times \R^2)\right) = \E_X\pp_\theta\left(\theta \in \left(\frac{1}{|X|}A\right)\times \R^2\right),
\]
since for dilates of the set $A \times \R^2$, we plainly have $\lambda(A\times \R^2) = (\lambda A) \times \R^2$, $\lambda > 0$. Using Lemma \ref{lm:Archimedes} again, and a change of variables, we obtain
\begin{align*}
\pp_\theta\left(\theta \in \left(\frac{1}{|X|}A\right)\times \R^2\right) &= \pp_\theta\left(P(\theta) \in \left(\frac{1}{|X|}A\right)\right) \\
&= \frac{1}{\omega_d}\int_{\R^d}\1_{\left\{x \in A/|X|, |x| \leq 1\right \}} \dd x \\
&= \frac{1}{\omega_d}\int_{A}|X|^{-d}\1_{\left\{|x| \leq |X|\right \}} \dd x.
\end{align*}
Consequently,
\[
\p{Y \in A} = \frac{1}{\omega_d}\int_{A}\Big(\E|X|^{-d}\1_{\left\{|X| \geq |x|\right\}}\Big)\dd x.
\]
This means that $Y$ has density on $\R^d$ given by $p(x) = \frac{1}{\omega_d}\E|X|^{-d}\1_{\left\{|X| \geq |x|\right\}}$.
\end{proof}

We mention in passing that, alternatively, Lemma \ref{lm:density} can also be derived from a result of Baernstein II and Culverhouse, (6.5) in \cite{BC}.

\begin{lemma}\label{lm:densities}
The random variables $|\sum_{j=1}^n a_j\xi_j|$ and $|\sum_{j=1}^n a_jU_j|$ have densities, say, $f:[0,+\infty)\to \R$ and $g: [0,+\infty)\to \R$, respectively, which satisfy
\[
g(r) = dr^{d-1}\int_r^\infty s^{-d}f(s) \dd s, \qquad r \geq 0.
\]
\end{lemma}

Their proof relies on the Fourier inversion formula and a subtle calculation. Curiously, going the other way around, Lemma~\ref{lm:densities} can be readily obtained from Lemma~\ref{lm:density}. We sketch the argument for completeness.

\begin{proof}
Let $X = \sum_{j=1}^n a_j\xi_j$, $Y = \sum_{j=1}^n a_jU_j$. To see that $|X|$ has a density, let $S = \sum_{j=1}^{n-1} a_j\xi_j$ and note that $|X|^2 = |S|^2 + 2|S|a_n\theta + a_n^2$, where $\theta$ has the distribution of, say, the first coordiante of $\xi_n$ and is independent of $S$. Thus $|X|^2$ has a density. By Lemma \ref{lm:density}, the density $p$ of $Y$ is given by
\[
p(x) = \frac{1}{\omega_d}\E|X|^{-d}\1_{\left\{|X| \geq |x|\right\}} =  \frac{1}{\omega_d}\int_{|x|}^\infty s^{-d}f(s) \dd s.
\]
Integration in polar coordinates finishes the argument.
\end{proof}

\subsection{Probabilistic bounds}

We will use the following bounds established by K\"onig and Rudelson, see Propositions 5.1 and 5.4 in \cite{KR}.

\begin{proposition}[K\"onig-Rudelson, \cite{KR}]\label{prop:KR}
We have,
\begin{equation}\label{eq:KR1}
 \p{ \left|\sum_{j=1}^na_j\xi_j\right| \ge 1} \ge 0.1,
\end{equation}
and for $t > 1$,
\begin{equation}\label{eq:KR2}
\p{\left|\sum_{j=1}^n a_j\xi_j\right|\ge t } \le t^{d+2} \exp{\left( \frac{d+2}{2}(1-t^2)\right)}.
\end{equation}
\end{proposition}

Note that under our normalisation, $\E\left|\sum_{j=1}^na_j\xi_j\right|^2 = 1$. Inequality \eqref{eq:KR2} quantifies the strong concentration of $\sum_{j=1}^na_j\xi_j$. Bound \eqref{eq:KR1} is of anti-concentration type; it is sometimes referred to as Stein property, see \cite{Bur}, can robustly be approached by moment estimates (Paley-Zygmund-type inequalities), see \cite{Ver}, and has been very well studied for random signs, see \cite{DK, Ole}; for a generalisation to matricial coefficients, see Theorem 2 in \cite{CLT}.

We are now ready to prove Theorem \ref{thm:low-bd}. Since we do not try to optimise the values of constants involved, we forsake potentially more precise calculations in favour of simplicity of the ensuing arguments.

\begin{proof}[Proof of Theorem \ref{thm:low-bd}]
We fix $x \in B_2^d$ and let $X = |\sum a_j\xi_j|$. By Lemma \ref{lm:density}, we want to lower bound
\[
p(x) = \frac{1}{\omega_d}\E\left[X^{-d}\1_{\left\{X > |x|\right\}}\right].
\]
Crudely,
\[
\E\left[X^{-d}\1_{\left\{X > |x|\right\}}\right] \geq 2^{-d}\p{|x| < X < 2} \geq 2^{-d}\p{1 < X < 2},
\]
and by Proposition \ref{prop:KR}, 
\begin{align*}
\p{1 < X < 2} = \p{X \geq 1} - \p{X \geq 2} &\geq 0.1 - (2e^{-3/2})^{d+2}\\ &\geq 0.1 - (2e^{-3/2})^{3} > 0.01,
\end{align*}
thus finishing the proof.
\end{proof}

\section{Noncentral sections on effective support: Proofs of Theorem \ref{thm:low-bd-at-sqrt3} and Corollary \ref{cor:noncent}}

\subsection{Proof of Theorem \ref{thm:low-bd-at-sqrt3}}
Employing the localisation method of degrees of freedom for log-concave functions developed by Fradelizi and Gu\'edon in \cite{FG}, it suffices to prove the theorem for densities of the following form
\begin{equation}\label{eq:density-f}
f(x) = c\bigg(\1_{[0,a]}(|x|) + e^{-\gamma(|x|-a)}\1_{[a,a+b]}(|x|)\bigg), \qquad x \in \R,
\end{equation} 
where $a, b \geq 0$ not both $0$, $\gamma \geq 0$ and $c$ is determined by $\int_{\R} f = 1$. We refer to \cite{MNT} for the details of the argument. The difference is that \cite{MNT} deals with the minimisation of the entropy $f \mapsto -\int f \log f$ instead of the functional $f \mapsto f(\sigma \sqrt{3})$ under the constraint that $\sigma$ is fixed. 

{\red 
\begin{remark}\label{rem:caveat}
A little caveat is that our functional is \emph{not} strictly convex, but linear in $f$, whereas the argument in \cite{MNT} uses strict convexity. This can be rectified by considering, say, a functional $\Phi_\varepsilon(f) = \sigma f(\sigma\sqrt3) + \varepsilon\int f^2$, $\varepsilon > 0$.  Then \cite{MNT} gives that $\inf_{f \in \mathcal{F}_\sigma} \Phi_\varepsilon(f) = \inf_{f \in \tilde{\mathcal{F}}_\sigma} \Phi_\varepsilon(f)$, where $\mathcal{F}_\sigma$ are all log-concave densitities $f$ with variance $\sigma$ and $\tilde{\mathcal{F}}_\sigma$ are all such densities of the form \eqref{eq:density-f}. Below we shall show that $\inf_{f \in \tilde{\mathcal{F}}_\sigma} \Phi_0(f) = \frac{1}{\sqrt2}e^{-\sqrt6}$. Hence, 
\[
\inf_{f \in \mathcal{F}_\sigma} \Phi_\varepsilon(f) = \inf_{f \in \tilde{\mathcal{F}}_\sigma} \Phi_\epsilon(f) \geq  \inf_{f \in \tilde{\mathcal{F}}_\sigma} \Phi_0(f) = \frac{1}{\sqrt2}e^{-\sqrt6},
\]
so for every $f \in \mathcal{F}_\sigma$, $\Phi_0(f) = \Phi_\varepsilon(f) - \varepsilon\int f^2 \geq \frac{1}{\sqrt2}e^{-\sqrt6} - \varepsilon\int f^2$, and after letting $\varepsilon \to 0$, we also obtain $\Phi_0(f) \geq \frac{1}{\sqrt2}e^{-\sqrt6}$ for every $f \in \mathcal{F}_\sigma$, which is the desired conclusion from Theorem \ref{thm:low-bd-at-sqrt3}.
\end{remark}
}

Now we shall prove that $\Phi_0(f) \geq \frac{1}{\sqrt2}e^{-\sqrt6}$, for every $f\in \tilde{\mathcal{F}}_\sigma$.  First note that the functional $f \overset{\Phi_0}{\mapsto} \sigma f(\sigma\sqrt{3})$ is invariant under replacing $f(\cdot)$ with $\lambda f(\lambda \ \cdot )$ for any $\lambda > 0$. Therefore, it suffices to only consider $\gamma = 1$ (the case $\gamma = 0$ is formally contained in the case $b= 0$ with any $\gamma > 0$). For $f$ as above, we have
\[
1 = \int_{\R} f = 2c(a+ 1 - e^{-b})
\]
and
\[
\sigma^2 = 2c\left(\frac{a^3}{3} + \int_0^b (x+a)^2e^{-x}\dd x\right).
\]
For $a, b \geq 0$ not both $0$, we define
\[
A = A(a,b) = a + 1 - e^{-b}, \qquad B = B(a,b) = \sqrt{\frac{a^3 + 3\int_0^b (x+a)^2e^{-x}\dd x}{A}},
\]
so that $B = \sigma\sqrt{3}$ and $c = \frac{1}{2A}$. 

\textbf{Claim 1.} For all $a, b \geq 0$ not both $0$, we have $B(a,b) \geq A(a,b)$. In particular, $B(a, b) \geq a$ and $B(a, b) \geq 1-e^{-b}$.
\begin{proof}
The claim is equivalent to $AB^2 - A^3 \geq 0$. Note that for a fixed $a > 0$,
\[
\partial_b(AB^2 - A^3) = 3(a+b)^2e^{-b} - 3A^2e^{-b} = 3e^{-b}(a+b - A)(a+b+A) 
\] 
which is positive for every $b>0$, since $a+b - A = b - 1  + e^{-b} > 0$. Thus
\[
AB^2 - A^3 \geq (AB^2 - A^3)|_{b=0} = 0.\qedhere
\]
\end{proof}

By Claim 1, $B \geq a$, so when evaluating $f$ at $B$, we take the exponential bit of $f$, that is $f(B) = ce^{-(B-a)} = \frac{1}{2A}e^{a-B} = \frac{1}{2A}e^{A-1+e^{-b}-B}$
and \eqref{eq:f-at-sigma} becomes
\[
\frac{B}{A}e^{A-1+e^{-b}-B} \geq \sqrt{6}e^{-\sqrt{6}}.
\]
We introduce the function
\[
\psi(x) = x-1-\log x, \qquad x > 0,
\]
as it will be convenient to rewrite the last inequality equivalently by taking the logarithms of both sides,
\begin{equation}\label{eq:A-B}
\psi(B) \leq e^{-b}+\psi(A) + \psi(\sqrt{6}).
\end{equation}
Let $h(a,b)$ be the difference between the right hand side and the left hand side,
\[
h(a, b) = e^{-b}+\psi(A) + \psi(\sqrt{6}) - \psi(B).
\]
The proof is concluded through the following two claims.\qed

\textbf{Claim 2.} For every $a > 0$, $b \mapsto h(a,b)$ is nonincreasing on $(0, +\infty)$.

\textbf{Claim 3.} For every $a > 0$, we have $\lim_{b \to \infty} h(a,b) \geq 0$ with equality if and only if $a = 0$.

\begin{proof}[Proof of Claim 2.]
We fix $a > 0$ and differentiate with respect to $b$. We have,
\[
\partial_bA = e^{-b}
\]
and, using $B^2 = \frac{a^3 + 3\int_0^b(x+a)^2e^{-x}\dd x}{A}$,
\[
2B\partial_b B = \frac{3(a+b)^2e^{-b}}{A} - \frac{a^3 + 3\int_0^b(a+x)^2e^{-x}\dd x}{A^2}e^{-b} = \frac{e^{-b}}{A}\left(3(a+b)^2 - B^2\right).
\]
Plainly, $\psi'(x) = 1 - \frac{1}{x}$. Thus,
\[
e^{b}\partial_b h = -1 + \psi'(A) - \psi'(B)e^{b}\partial_bB = -\frac{1}{A} - \left(1 - \frac{1}{B}\right)\frac{1}{2AB}\left(3(a+b)^2 - B^2\right).
\]
Since $A > 0$, $\partial_b h \leq 0$ is therefore equivalent to the inequality
\[
(1-B)\big(3(a+b)^2-B^2\big) \leq 2B^2.
\]
We observe that $3(a+b)^2-B^2 \geq 0$. Indeed,
\[
AB^2 = a^3 + 3\int_0^b (a+x)^2e^{-x} \dd x \leq 3(a+b)^2a + 3(a+b)^2(1-e^{-b}) = 3(a+b)^2A.
\]
As a result, if $B \geq 1$, we conclude that $\partial_b h \leq 0$. When $B < 1$, $\partial_b h \leq 0$ is equivalent to the inequality
\[
B^2 \geq 3(a+b)^2\left(1 + \frac{2}{1-B}\right)^{-1}.
\]
The right hand side as a function of $B \in (0,1)$ is plainly decreasing. Using the bound $B \geq 1-e^{-b}$ from Claim 1, it thus suffices to show that
\[
B^2 \geq 3(a+b)^2\left(1 + 2e^{b}\right)^{-1},
\]
or, equivalently,
\[
AB^2 - 3A(a+b)^2\left(1 + 2e^{b}\right)^{-1} \geq 0.
\]
We fix $a > 0$. There is equality at $b= 0$. We take the derivative in $b$ of the left hand side which reads
\begin{align*}
3(a+b)^2e^{-b} &- 3e^{-b}(a+b)^2\left(1 + 2e^{b}\right)^{-1} - 6A(a+b)\left(1 + 2e^{b}\right)^{-1} \\
&+ 6A(a+b)^2\left(1 + 2e^{b}\right)^{-2}e^b \\
&=\frac{6(a+b)^2}{1+2e^b}\left(1 - \frac{A}{a+b} + \frac{Ae^b}{1+2e^b}\right).
\end{align*}
Clearly, $a+b \geq a + 1-e^{-b} = A$. Consequently the above expression is positive, which finishes the proof.
\end{proof}

\begin{proof}[Proof of Claim 3.]
We readily have,
\begin{align*}
A(a,\infty) &= a+1, \\
B(a,\infty) &= \sqrt{\frac{a^3 + 3(a^2+2a+2)}{a+1}} = \sqrt{\frac{(a+1)^3 + 3(a+1) + 2}{a+1}}.
\end{align*}
As a result, setting $x = a+1$ and 
\[
f(x) = \sqrt{x^2 + 3 + \frac2x}, \qquad x \geq 1,
\]
we obtain
\[
h(a,\infty) = \psi(x) + \psi(\sqrt{6}) - \psi\big(f(x)\big),
\]
where, recall, $\psi(u) = u - 1 - \log u$.
Note that the right hand side vanishes at $x=1$. To conclude, we show that its derivative is positive for every $x > 1$. The derivative reads
\begin{align*}
1 - \frac{1}{x} - \left(1 - \frac{1}{f(x)}\right)f'(x) &= 1 - \frac{1}{x} - \frac{f(x)-1}{f(x)^2}\left(x-\frac{1}{x^2}\right) \\
&= \frac{x-1}{x}\cdot\frac{f(x)-1}{f(x)^2}\left(\frac{f(x)^2}{f(x)-1} - x- 1 - \frac{1}{x} \right)
\end{align*}
Plainly, $f(x) > 1$. Moreover,
\[
\frac{f(x)^2}{f(x)-1} > f(x)+1 = \sqrt{x^2 + 3 + \frac{2}{x}} + 1 > \sqrt{x^2 + 2 + \frac{1}{x^2}} + 1 = x + \frac{1}{x}  +1,
\]
which shows that the derivative is positive and finishes the proof.
\end{proof}

\subsection{Proof of Corollary \ref{cor:noncent}}

This is a standard argument. We fix a unit vector $\theta$ in $\R^d$ and consider the section function by hyperplanes orthogonal to $\theta$,
\[
f(t) = \vol_{d-1}(K \cap (t\theta + \theta^\perp)), \qquad t \in \R.
\]
By the Brunn-Minkowski inequality, this defines a log-concave function. Since $K$ is symmetric, $f$ is even. In particular, it is nonincreasing on $[0,+\infty)$, so it suffices to show that $L_Kf(L_K\sqrt{3}) \geq \frac{1}{\sqrt{2}}e^{-\sqrt{6}}$. Since $K$ is of volume $1$, with isotropic constant $L_K$, we have $\int_{\R} f = 1$ and $\sigma = \sqrt{\int_{\R} t^2f(t) \dd t} = \sqrt{\int_{K} \scal{x}{\theta}^2 \dd x} =  L_K$. Theorem \ref{thm:low-bd-at-sqrt3} yields the result.

To see that Corollary \ref{cor:noncent} is indeed sharp, we present the following construction confirming Remark \ref{rem:noncent-sharp}.

\subsection{Proof of Remark \ref{rem:noncent-sharp}}

Given $\lambda = (\lambda_1, \lambda_2) \in (0,\infty)^2$, we define a double cone $K_\lambda$ in $\mathbb{R}^{d+1}$ by
\[
    K_\lambda = \left\{ (x,t) \in \R^d\times \R, \ |t| \leq \lambda_2 d, \ \ |x| \leq \lambda_1 \left( 1 - \frac{|t|}{\lambda_2 d} \right)
    \right\}.
\]
We have by direct computation
\[
\vol(K_\lambda) =  \frac{2 d \,\omega_d \,\lambda_1^d \lambda_2}{d+1}.
\]
Setting
\begin{align*}
    \lambda_1 &= L_d \sqrt{\frac{(d+2)(d+3)}{d+1}},
        \\
    \lambda_2 &= L_d \sqrt{\frac{(d+2)(d+3)}{2d^2} },
\end{align*}
with
\[
    L_d  = \left( \frac{ (d+1)^{d+2}}{2((d+3)(d+2))^{d+1} \omega_d^2} \right)^{\frac{1}{2(d+1)}},
    \]
the body $K_\lambda$ is in isotropic position with isotropic constant $L_d$, and, in particular, $\vol_{d+1}(K_\lambda) = 1$.  Moreover,
\[
   \frac{L_d}{\lambda_2} \xrightarrow{d \to \infty} \sqrt{2}
\]
and 
\[
\omega_d\lambda_1^{d+1} = \sqrt{\frac{d+1}{2}}.
\]
Thus, 
\begin{align*}
    L_d \vol_d( K_\lambda \cap \{ t = L_d \sqrt{3} \})
        &=
            L_d \,\omega_d \lambda_1^d  \left( 1 - \frac{L_d \sqrt{3}}{\lambda_2 d } \right)^d
                \\
        &=
           \sqrt{\frac{d+1}{(d+2)(d+3)}}\omega_d\lambda_1^{d+1} \left( 1 - \frac{L_d \sqrt{3}}{\lambda_2 d } \right)^d
                \\
        &=
           \sqrt{\frac{(d+1)^2}{2(d+2)(d+3)}}
            \left( 1 - \frac{L_d \sqrt{3}}{\lambda_2 d } \right)^d.
\end{align*}
Taking the limit we see that
\[
    \lim_{d \to \infty} L_d \vol(K_\lambda \cap \{ t = L_d \sqrt{3} \}) = \frac 1 {\sqrt{2}} e^{-\sqrt{6}}.
\]

\begin{remark}\label{rem:interpol}
Given $t_0 \in [0,\sqrt{3}]$, consider the problem
\begin{equation}\label{eq:inf-at-t0}
\inf\{ f(t_0), \ \ f \text{ is an even log-concave density on $\R$ with variance $1$}\}.
\end{equation}
Theorem \ref{thm:low-bd-at-sqrt3} asserts that at $t_0 = \sqrt{3}$ the infimum equals $\frac{1}{\sqrt{2}}e^{-\sqrt{6}}$ and is attained for the symmetric exponential density. It is a well-known result going back to Moriguti's work \cite{Mor} that for an arbitrary probability density $f$ on $\R$, we have $\|f\|_\infty^2 \geq \frac{1}{12}\left(\int_{\R} x^2f(x) \dd x\right)^{-1}$, with equality attained for symmetric uniform densities. As a result, when specialised to even log-concave densitites $f$ of variance~$1$, we get that at the point $t_0 = 0$ \eqref{eq:inf-at-t0} equals $\frac{1}{2\sqrt{3}}$ and is attained for the symmetric uniform density. 

{\red For a fixed $t_0 \in (0,\sqrt3)$, 
we again get from the proof of Theorem \ref{thm:low-bd-at-sqrt3} that this infimum is attained at a density of the form \eqref{eq:density-f}. However, we do not have a good prediction for the exact form of the minimizer (explicit numerical calculations show that for every $t_0 \geq 0.7$, the minimiser is neither uniform nor exponential which are outperformed e.g. by a truncated exponential density).}
\end{remark}

\section{R\'enyi entropy: Proof of Theorem \ref{thm:Renyi-anti-conc}}

The argument simply relies on combining Theorem \ref{thm:low-bd} with the following subadditivity result for R\'enyi entropies extending the classical entropy power inequality.

\begin{theorem}[Bobkov-Chistyakov, \cite{BCh-Renyi}]\label{thm:BCh-REPI}
Let $p \geq 1$. For independent random variables $X_1, \dots, X_n$ in $\mathbb{R}^d$, we have
\[
    N_p(X_1 + \cdots + X_n) \geq e^{-1} \sum_{i=1}^n N_p(X_i).
\]
\end{theorem}

In fact, they obtained the better constant $c_p = e^{-1}p^{1/(p-1)}$ in place of $e^{-1}$. Moreover, as established in \cite{MMX}, the case $p = \infty$ admits an optimal dimensionally dependent constant $c_\infty(d)= \frac{\Gamma(\frac d 2 + 1)^{2/d}}{\frac d 2 +1}$, for $d \geq 2$.  However, for simplicity of expression, and as our other computations do not attempt to approach optimal constants, nor do the larger constants affect the asymptotics of corollaries to come, we will not make use of this sharpening.

\begin{proof}[Proof of Theorem \ref{thm:Renyi-anti-conc}]
By Theorem \ref{thm:low-bd} and \eqref{eq:const-c_d}, the density function of $\sum_{j=1}^n \lambda_jU_j$ is bounded below by $\frac{1}{100\cdot 2^d}$ times the density function of $U_0$.  It follows that $\sum_{j=1}^n (X_j + \lambda_j U_j)$ has a density function bounded \emph{pointwise} below by $\frac{1}{100\cdot 2^d}$ times the density function of $S+U_0$. Thus,
    \[
        N_p(S+ U_0) \geq (100\cdot 2^d)^{-\frac{2p}{d(p-1)}} N_p\left( \sum_{j=1}^n (X_j + \lambda_j U_j) \right).
    \]
    By Theorem \ref{thm:BCh-REPI},
    \[
        N_p\left( \sum_{j=1}^n (X_j + \lambda_j U_j) \right) \geq e^{-1} \sum_{j=1}^n N_p(X_j + \lambda_j U_j).
    \]
    Combining the two inequalities yields the result with 
\[
C_{p,d} = e\cdot(100\cdot 2^d)^{\frac{2p}{d(p-1)}} < e\cdot 2^{\frac{2p}{p-1}\frac{d+7}{d}}.
\]
The same argument can be applied with sharpened constants in the case~$p = \infty$.
\end{proof}

\begin{proof}[Proof of Corollary \ref{cor:anti-conc}]
By homogeneity, we can assume that $\sum_{j=1}^n \lambda_j^2 = 1$. When $\lambda = 1$, in view of \eqref{eq:Q-N}, the corollary follows immediately with constant $(e \cdot {\red 2^{2\frac{d+7}{d}}})^{d/2} = e^{d/2}{\red 2^{d+7}}$ by setting $p=\infty$ in Theorem \ref{thm:Renyi-anti-conc}. When $\lambda \geq 1$, using the union bound, we get $Q_X(\lambda) \leq (2\lambda + 1)^dQ_X(1)$ because by a standard volumetric argument a ball of radius $\lambda \geq 1$ can be covered by at most $(2\lambda+1)^d$ unit balls (see, e.g. Theorem 4.1.13 in \cite{AGM}), and the corollary follows from the previous case.
\end{proof}

\section{Reversals under log-concavity}\label{sec:reversals}

It turns out that in the one dimensional case, the variance of a log-concave random variable $X$ is a \emph{good proxy} for its maximum functional, more precisely $\frac{1}{12} \leq \Var(X)M(X)^2 \leq 1$, see Proposition 2.1 in \cite{BCh}. Building on this and the additivity of variance under independence, Bobkov and Chistyakov (\cite{BCh}, Corollary 2.2) derived two-sided matching bounds on the concentration function of sums of independent log-concave random variables.

In higher dimensions, such a proxy with \emph{good tensorisation} properties seems to be a holy grail. If, however, we restrict our attention to isotropic random vectors, that is the centred ones with identity covariance matrix, then the maximum functional is directly related to the isotropic constant, which is well-studied in geometric functional analysis (see e.g. \cite{AGM, BGVV}). 

More specifically, if $X$ is a random vector in $\R^d$, its isotropic constant $L_X$ is defined to be
\[
L_X = (\det[\text{Cov}(X)])^{\frac{1}{2d}}M(X)^{\frac{1}{d}}
\]

By a standard argument, $L_X \geq \kappa_d$, where $\kappa_d$ is the isotropic constant of a random vector uniform on the unit volume Euclidean ball $\omega_d^{-1/d}B_2^d$. Moreover, $\kappa_d \geq \frac{1}{12}$. Let 
\[
\mathscr{K}_d = \sup L_X,\]
be the supremum taken over all log-concave isotropic random vectors $X$ in $\R^d$. Bourgain's famous slicing conjecture originating in \cite{Bo} asked whether $\mathscr{K}_d$ is upper-bounded by a universal constant, see also \cite{AGM, BGVV, KM}. {\red A series of recent breakthroughs culminated in the work of Klartag and Lehec \cite{KlaLeh}, who confirm the slicing conjecture. They prove that, indeed, $\mathscr{K}_d = O(1)$. For an outlook of the recent advancements, see references in \cite{KlaLeh}, as well as a follow-up work of Bizeul \cite{Biz}.}

Since covariance matrices add up for sums of independent random variables, we get two-sided bounds as in the one-dimensional case.

\begin{theorem}\label{thm:Q-log-conc}
Let $X_1, \dots, X_n$ be independent log-concave random vectors in $\R^d$ and $S$ be their sum. For $\lambda > 0$, we have
\[
\frac{c^d\omega_d\lambda^d}{\left(\det\left[\frac{\lambda^2}{{\red d+2}}I + \sum_{j=1}^n \mathrm{Cov}(X_j)\right]\right)^{1/2}} \leq Q_S(\lambda) \leq \frac{C^d\omega_d\lambda^d}{\left(\det\left[\frac{\lambda^2}{{\red d+2}}I + \sum_{j=1}^n \mathrm{Cov}(X_j)\right]\right)^{1/2}},
\]
{\red where $c, C > 0$ are universal constants.}
\end{theorem}
\begin{proof}
For a log-concave random vector $X$, by the definitions of the isotropic constant and constants $\kappa_d$ and $\mathscr{K}_d$, we have
\[
\frac{\kappa_ d^d}{\sqrt{\det[\mathrm{Cov}(X)]}} \leq M(X) \leq \frac{\mathscr{K}_ d^d}{\sqrt{\det[\mathrm{Cov}(X)]}}.
\]
{\red Moreover, as discussed above, $c \leq \kappa_d \leq \mathscr{K}_d \leq C$ for some universal positive constants $c$ and $C$.}
Crucially, sums of independent log-concave random vectors are log-concave and uniform distributions on convex sets are log-concave. Therefore, we can apply this double-sided bound to $X = S + \lambda U$, where $U$ is a random vector uniform on the unit ball independent of the $X_j$'s. Using \eqref{eq:Q-M} and noting that 
\[
\mathrm{Cov}(X) = \mathrm{Cov}(\lambda U) + \sum_{j=1}^n \mathrm{Cov}(X_j) = \lambda^2{\red \frac{1}{d+2}}I + \sum_{j=1}^n \mathrm{Cov}(X_j),
\]
where $I$ stands for the $d \times d$ identity matrix, we arrive at the desired bounds.
\end{proof}

\end{document}